\documentclass{article}
\usepackage[utf8]{inputenc}
\usepackage{amssymb, eurosym}
\usepackage{amsmath}
\usepackage{amsthm}

\usepackage[includeheadfoot,a4paper]{geometry}

\usepackage{amsfonts}

\usepackage[all]{xy}

\theoremstyle{definition}
\newtheorem{defi}{Definition}[section]
\newtheorem{lem}[defi]{Lemma}
\newtheorem{teo}[defi]{Theorem}
\newtheorem{cor}[defi]{Corollary}
\newtheorem{prop}[defi]{Proposition}
\newtheorem{ex}[defi]{Example}

\theoremstyle{remark}
\newtheorem{obs}[defi]{Remark}

\DeclareMathOperator{\rad}{\operatorname{rad}}
\DeclareMathOperator{\md}{\operatorname{mod}}
\DeclareMathOperator{\ind}{\operatorname{ind}}
\DeclareMathOperator{\ff}{\mathcal{F}}
\DeclareMathOperator{\Dp}{\operatorname{dp}} 
\DeclareMathOperator{\Id}{\operatorname{Id}} 
\DeclareMathOperator{\Hom}{\operatorname{Hom}} 
\DeclareMathOperator{\Ker}{\operatorname{Ker}}

\DeclareMathOperator{\Ga}{\Gamma}
\DeclareMathOperator{\lra}{\longrightarrow}
\DeclareMathOperator{\Rr}{\mathfrak{R}}

\title{Mesh-comparable components of the Auslander-Reiten quiver}
\author{Viktor Chust \thanks{corresponding author: viktorchust.math@gmail.com}, Flávio U. Coelho \\ Inst. of Mathematics and Statistics, Univ. of São Paulo (IME-USP)}
\date{}

\begin{document}

\maketitle

\begin{abstract}
The idea of using Riedtmann's well-behaved functors to study compositions of irreducible morphisms has been explored in a number of articles. Here we introduce the concept of mesh-comparable components of the Auslander-Reiten quiver, which are components for which a Riedtmann functor exists without the necessity of taking a covering, such as the universal or the generic one. We show properties of this type of component, and study the problem of compositions of irreducible morphisms in this context.  

\vspace{1ex}

{\it \noindent Keywords: Riedtmann functors, well-behaved functors, mesh category, compositions of irreducible morphisms

\noindent Mathematics Subject Classification (MSC 2020): Primary 16G70, Secondary 16G10}
\end{abstract}

\section*{}
Given a finite-dimensional $k$-algebra $A$ (where $k$ is a field), one way of organizing the category mod$A$ of the finitely generated right $A$-modules is by the so-called Auslander-Reiten quiver $\Gamma$(mod$A)$. The vertices of this quiver correspond to the isomorphism classes of the indecomposable objects in mod$A$ and the arrows indicate the existence of irreducible morphisms between them. Recall that, for two $A$-modules $X,Y$, a morphism $f \colon X \lra Y$ is {\it irreducible} if it does not split and if for any decomposition $f = gh$ then either $g$ is a split epimorphism or $h$ is a split monomorphism. 

By using {\it well-behaved functors}, we have a tool for studying Auslander-Reiten components. These functors were introduced in the related works of C. Riedtmann \cite{Rie} and K. Bongartz, P. Gabriel \cite{BG}. Although well-behaved functors have slightly different definitions throughout the literature, they usually have the form $F: k(\Delta) \rightarrow \ind \Gamma$, where $\ind \Gamma$ is the subcategory of indecomposable modules in a component $\Gamma$ of the Auslander-Reiten quiver, and $k(\Delta)$ is the {\it mesh category} over a {\it covering} $\Delta$ of the quiver $\Gamma$, usually the {\it universal covering} or the {\it generic covering}.

The idea behind well-behaved functors is that $k(\Delta)$ is a combinatorial representation of $\ind \Gamma$, and usually has a better behavior than the latter. So a functor $k(\Delta) \rightarrow \ind \Gamma$ with good properties allows a comparison between these two categories, thus enabling us to obtain results about one from the other.

Well-behaved functors $k(\Delta) \rightarrow \ind \Gamma$ can be shown to always exist for a certain covering $\Delta$ of $\Gamma$, at the price for us having to consider the covering quiver $\Delta$ which is in general different and more complicated than $\Gamma$. The question on whether functors of the form $k(\Gamma) \rightarrow \ind \Gamma$ can exist was briefly discussed by Riedtmann herself in the seminal work \cite{Rie}. Although such functors exist sometimes, there are examples (see, e.g., Example~\ref{ex:14bis} below) where they do not. 

Later works, mainly by S. Liu (\cite{Liu3}), which alludes to a certain {\it choice of irreducible morphisms} as a tool to prove the main theorem of that paper, and by C. Chaio, M. I. Platzeck, S. Trepode (\cite{CPT}), which use the terminologies {\it irreducible morphisms which satisfy the mesh-relations} and {\it paths in the mesh}, seem to suggest the use of well-behaved functors $k(\Gamma) \rightarrow \ind \Gamma$. Namely, we prove in Proposition~\ref{prop:mesh-c equiv escolha} below that the existence of such a functor is equivalent to a choice as prescribed by S. Liu in the work cited above.

That called for a more systematic approach to components having this kind of structure, and this is one of the main goals here. Accordingly, given a component $\Ga$ of the Auslander-Reiten quiver, whenever there exists a well-behaved functor $F: k(\Ga) \rightarrow \ind \Ga$, we will say that $\Ga$ is {\it mesh-comparable via $F$}. As we will see, not all components are mesh-comparable, and the concept of mesh-comparable, although related, is different from the well-known concept of {\it standard components}.

As a consequence of the good properties of Riedtmann's well-behaved functors, the chosen irreducible morphisms, as alluded to by \cite{Liu3} and studied in Section~\ref{sec:defi mesh-c} below, satisfy a very nice property with respect to their compositions. In particular, this is strongly related to the problem of compositions of irreducible morphisms, as we now quickly discuss. 

Let us rephrase the definition of irreducible morphisms as given above. Given two indecomposable $A$-modules $X$ and $Y$, we denote by $\rad_A(X,Y)$ the set of non-isomorphisms $X \rightarrow Y$. Clearly, one can extend it to general modules as follows: $\rad_A(\oplus_{i=1}^n X_i,\oplus_{j=1}^m Y_j) = \oplus_{i=1}^n\oplus_{j=1}^m \rad_A(X_i,Y_j)$. The morphisms in $\rad_A$ are called {\it radical morphisms}. Using the fact that $\rad_A$ is an ideal of the category $\md A$, one can consider its powers, defined recursively by: $\rad_A^0 = \Hom_A, \rad_A^1= \rad_A, \rad_A^n = \rad_A^{n-1} \cdot \rad_A$, where the product $\cdot$ stands for composition of morphisms. We also define $\rad_A^{\infty} = \cap_{n \geq 0} \rad_A^n$. So, a morphism $X \rightarrow Y$ with $X,Y$ indecomposable is irreducible if and only if it belongs to $\rad_A(X,Y) \setminus \rad_A^2(X,Y)$. As shown by Auslander-Reiten theory, irreducible morphisms generate any other morphism modulo $\rad^{\infty}$, what emphasizes the importance of irreducible morphisms.

Clearly, the composition of $n$ irreducible morphisms between indecomposable modules belongs to $\rad^n$, and one could wonder if it is also true that, provided it is non-zero, it does not belong to $\rad^{n+1}$. This is not true for $n \geq 2$, as shown in \cite{CCT1}. Indeed, this same paper shows an example of two irreducible morphisms whose composite is non-zero and belongs to $\rad^{\infty}$. Deciding in which cases there might be a non-zero composition of $n$ irreducible morphisms belonging to $\rad^{n+1}$ has become an interesting line of investigation. For example, \cite{CCT1} fully characterizes the case of two irreducible morphisms whose composite belongs to $\rad^3 \setminus \{0\}$.  
In mesh-comparable components, as we shall introduce below, if we take $n$ chosen irreducible morphisms (i.e., morphisms which were chosen for each arrow of $\Gamma$), then their composite either belongs to $\rad^n \setminus \rad^{n+1}$ or is zero. (This is our Corollary~\ref{cor:composicao_bem_comportada_chosen}).

From the compositions of chosen morphisms, we will show how to study compositions of general irreducible morphisms, by using a method based on decomposing a morphism in parts given by the filtration of the radical ideal. Thus mesh-comparable components are particularly suited for the study of composites of irreducible morphisms.

This paper is organized as follows. Section~\ref{sec:prelim} is devoted to Preliminaries, when we recall the main concepts used along the paper and state the proper notations. In Section~\ref{sec:defi mesh-c}, we define mesh-comparable components and establish basic properties of these components, with examples. The decomposition of a morphism in its parts is done in Section~\ref{sec:decomposition}. We apply this decomposition in Section~\ref{sec:new proof Liu}, to give a new proof of the well-known result by S. Liu that all standard components are generalized standard, and in Section~\ref{sec:mesh-c x compositions}, relating to the problem of compositions of irreducible morphisms. Section~\ref{sec:standard x mesh-c} brings one of our main results here, Theorem~\ref{th:gen st + mesh-comp}, which states that generalized standard mesh-comparable components are standard. We bring consequences and examples around this result. Finally, Section~\ref{sec:mesh-c x degrees} is devoted to a relation between mesh-comparability and the {\it degrees} of irreducible morphisms, a concept defined by S. Liu in \cite{Liu1}.

\section{Preliminaries}
\label{sec:prelim}

Along the paper, $k$ will denote an algebraically closed field and $A$ will denote an algebra, that is, a finite-dimensional associative and unitary $k$-algebra $A$. Also, we denote by $\md A$ the category of finitely generated right $A$-modules. For details in module theory, we refer the reader to \cite{AC2}.  

\subsection{Quivers}
\label{subsec:quiv}

A \textbf{quiver} $Q$ is given by a 4-uple $Q=(Q_0,Q_1,s,e)$, where $Q_0, Q_1$ are sets and $s,e:Q_1 \rightarrow Q_0$ are functions. The elements of $Q_0$ are called the  {\bf vertices} of $Q$, while those of $Q_1$ are called the {\bf arrows} of $Q$. Also, given an arrow $\alpha \in Q_1$, the vertices $s(\alpha)$ and $e(\alpha)$ are called, respectively, the {\bf start vertex} and the {\bf end vertex} of $\alpha$. A {\bf path of length $n$ in $Q$} is given by a sequence of $n$ arrows $\beta_n \cdots \beta_1$, where, for every $1 \leq i \leq n$, $e(\beta_i) = s(\beta_{i+1})$ for every $1 \leq i < n$. Additionally, one associates to each vertex $x$ of $Q$ a trivial \textbf{path of length 0}, denoted by $\epsilon_x$. Of course, $s(\epsilon_x)=e(\epsilon_x) =x$.

\subsection{Auslander-Reiten theory}
\label{subsec:ar}

We have recalled the definition of {\it radical} and {\it irreducible} morphisms at the introduction, and we use this section to give further reminders on Auslander-Reiten theory. However, for concepts of this theory which are not covered in this section, we indicate \cite{AC}.

A morphism $f: X \rightarrow Y$ between two $A$-modules is called {\bf left minimal} if every endomorphism $h: Y \rightarrow Y$ such that $hf = f$ is invertible. The morphism $f$ is called {\bf left almost split} if it is radical, the module $X$ is indecomposable and if for every radical morphism $u: X \rightarrow U$ there is a $u':Y \rightarrow U$ such that $u = u'f$. A morphism is called a {\bf source morphism} if it is both left minimal and left almost split. Dually, one can define {\bf right minimal} and {\bf right almost split} morphisms, and a morphism which is both of these things is called a {\bf sink morphism}.

A short exact sequence $0 \rightarrow L \xrightarrow{f} M \xrightarrow{g} N \rightarrow 0 $ in mod$A$ is called an \textbf{almost split sequence} (or an \textbf{Auslander-Reiten sequence}) provided both $f$ and $g$ are irreducible morphisms. 

If $0 \rightarrow L \xrightarrow{f} M \xrightarrow{g} N \rightarrow 0 $ is an Auslander-Reiten sequence, then it does not split, $L$ and $N$ are indecomposable modules, $f$ is a source morphism and $g$ is a sink morphism.

A key result of Auslander-Reiten theory is the theorem on the existence and uniqueness of Auslander-Reiten sequences, as follows: given a non-projective indecomposable module $N$, there must be an Auslander-Reiten sequence $0 \rightarrow \tau N \rightarrow M \rightarrow N \rightarrow 0$, which is uniquely determined by $N$ up to isomorphism of short exact sequences. Dually, if $L$ is indecomposable not injective, there is an Auslander-Reiten sequence $0 \rightarrow L \rightarrow M \rightarrow \tau^{-1} L \rightarrow 0$, which is uniquely determined by $L$. Actually, $\tau$ and its inverse $\tau^{-1}$ above give rise to functors, defined between stable quotients of the module category. The functor $\tau$ is called the {\bf Auslander-Reiten translation}.

The \textbf{Auslander-Reiten quiver} $\Gamma(\md A)$ of the algebra $A$ is a quiver defined as follows: (i) the vertices of  $\Gamma(\md A)$ are in bijection with the isomorphism classes of indecomposable $A$-modules; (ii) given vertices $M, N$, the number of arrows $M\rightarrow N$ in $\Gamma(\md A)$ equals the dimension of the $k$-vector space $\operatorname{irr}_k(M,N) \doteq \rad_A(M,N)/\rad^2_A(M,N)$ of the irreducible morphisms from $M$ to $N$. Clearly, the Auslander-Reiten translation $\tau$ acts on the vertices of $\Gamma(\md A)$. 

Connected components of the Auslander-Reiten quiver will be called here as {\it Auslander-Reiten components}, for short. Also, if $\Gamma$ is an Auslander-Reiten component of $A$, then $\ind \Gamma$ will stand for the subcategory of $\md A$ determined by all indecomposable modules which belong to $\Gamma$.

\subsection{Translation quivers}
\label{subsec:tr quivs}

Translation quivers are a generalization of Auslander-Reiten quivers, which takes into account the combinatorial structure based on the Auslander-Reiten translation.
Namely, a \textbf{translation quiver} is a quiver $\Gamma$ without loops (that is, no arrows starting and ending at the same vertex) such that: (i) $\Gamma$ has a set of vertices called {\it projective vertices} and another set of vertices called {\it injective vertices}; (ii)  there is a bijection (called \textit{translation}) $\tau \colon  x \mapsto \tau(x)$ between non-projective and non-injective vertices; and (iii) for each pair of vertices $x$ and $y$ of $\Gamma$, with $x$ non-projective, there is a bijection $\sigma: \alpha \mapsto \sigma(\alpha)$ between arrows of the form $y \rightarrow x$ and arrows of the form $\tau(x) \rightarrow y$. In addition, we shall assume that translations quivers are locally finite, that is, there are only finitely many arrows starting or ending at each vertex. In general, we will always use the letters $\tau$ e $\sigma$ to denote the functions associated to a given translation quiver $\Gamma$. 

Every component of an Auslander-Reiten quiver is an example of a translation quiver, with the Auslander-Reiten translation being the function $\tau$ above.

A path  $x_0 \xrightarrow{\alpha_1} x_1 \xrightarrow{\alpha_2} \cdots \xrightarrow{\alpha_n} x_n$ in a translation quiver $\Gamma$ is called {\bf sectional} provided  $x_i \neq \tau x_{i+2}$ for all $0 \leq i \leq  n-2$.

Finally, given a non-projective vertex $x$ of $\Gamma$, the full subquiver of $\Gamma$ determined by all the arrows that end at $x$ and the arrows that start at $\tau x$ is called the {\bf mesh} ending at $x$ (or starting at $\tau x$). If $\Gamma$ is an Auslander-Reiten quiver, then a mesh corresponds to an Auslander-Reiten sequence.

    
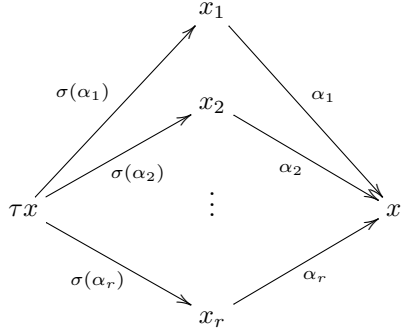
\begin{displaymath}
        \xymatrix{ && x_1 \ar[ddrr]^{\alpha_1}&& \\
        && x_2 \ar[drr]_{\alpha_2}&&\\
        \tau x \ar[uurr]^{\sigma(\alpha_1)} \ar[urr]_{\sigma(\alpha_2)} \ar[drr]_{\sigma(\alpha_r)}&& \vdots && x \\
        &&x_r \ar[urr]_{\alpha_r} &&}
    \end{displaymath}

\subsection{Riedtmann functors}

The {\it well-behaved functors}, as the literature usually has them, were originally introduced by Riedtmann in \cite{Rie} and right after in parallel by Bongartz and Gabriel (\cite{BG}). In our work we call these functors as {\it Riedtmann functors} since we consider a minor extension of the original concept and to give more credit for the creators. We also indicate \cite{CMT1,CMT2} for further generalizations of Riedtmann functors. In our survey \cite{CCsurvey}, we have included a comprehensive explanation of this concept.

We dedicate this section to recall these functors, but differently from the original sources, we do not define them over coverings of quivers, since that is not necessary here. We will follow an {\it ad hoc} approach in this paper.

We have fixed $k$ an algebraically closed field. Let $\Gamma$ be a translation quiver. With these data we define two categories:

The \textbf{path category} over $\Gamma$ is defined as the category $k\Gamma$ whose objects are the vertices of $\Gamma$ and the morphisms between two vertices $x$ and $y$ are the elements of the vector space given by formal linear combinations of paths over $\Gamma$ which go from $x$ to $y$. 

The \textbf{mesh category} over $\Gamma$, denoted by $k(\Gamma)$, is the quotient category of the path category $k\Gamma$ by the ideal generated by all {\bf mesh relations} in $\Gamma$, that is, elements of the form $\sum \alpha (\sigma \alpha): \tau x \rightarrow x$, where $x$ is a non-projective vertex of $\Gamma$ and the summation runs through all arrows $\alpha$ that end in $x$.

Note that the categories $k\Ga$ and $k(\Ga)$ are \textbf{$\mathbb{N}$-graded}, with the paths (or classes of paths) being homogeneous elements with degree given by their length, and the mesh relations being homogeneous elements of degree 2.

We can define the \textbf{radical of the mesh category} $k(\Ga)$ as being the ideal $\Rr k(\Ga)$ generated by the morphisms of degree 1, which are the equivalence classes of arrows in $\Gamma$. We also consider the powers of the ideal $\Rr k(\Ga)$, which are defined recursively: $\Rr^0 k(\Ga) = k(\Ga)$, $\Rr^1 k(\Ga) = \Rr k(\Ga)$, and for $n \geq 1$, $\Rr^n k(\Ga) = \Rr^{n-1} k(\Ga)\cdot \Rr k(\Ga)$. Also note that $\Rr^n k(\Ga)$ coincides with the ideal generated by the classes of the paths having length greater than or equal to $n$.

Although the definition of path and mesh categories make sense for general translation quivers, the definition of Riedtmann functors require $\Gamma$ to be a component of the Auslander-Reiten quiver.

\begin{defi}[\cite{Rie, BG, CMT1}]
Suppose $k$ is algebraically closed. Let $\Gamma$ be a component of the Auslander-Reiten quiver of $A$. A $k$-linear functor $F:k(\Gamma) \rightarrow \ind \Gamma$ is called a {\bf Riedtmann functor} if the following conditions are verified for every vertex $X$ of $\Gamma$:

\begin{enumerate}
    \item $FX = X$
    \item if $\alpha_1: X \rightarrow X_1, \ldots, \alpha_r: X \rightarrow X_r$ are all the arrows in $\Gamma$ that start at $X$, then $[F(\overline{\alpha_1}) \ldots F(\overline{\alpha_r})]^t: X \rightarrow X_1 \oplus \ldots \oplus X_r$ is a source morphism (i.e., a left minimal almost split morphism).
    
    \item if $\alpha_1: X_1 \rightarrow X, \ldots, \alpha_r: X_r \rightarrow X$ are all the arrows in $\Gamma$ that end at $X$, then $[F(\overline{\alpha_1}) \ldots F(\overline{\alpha_r})]: X_1 \oplus \ldots \oplus X_r \rightarrow X$ is a sink morphism (i.e, a right minimal almost split morphism).
\end{enumerate}
\end{defi}

\begin{obs}
For a Riedtmann functor $F:k(\Gamma) \rightarrow \ind \Gamma$, if $\alpha$ is an arrow of $\Gamma$, then $F(\overline{\alpha})$ is an irreducible morphism. (Actually this type of condition was used to define these functors in \cite{BG,Rie}, which only deal with algebras of finite type).

Reciprocally, only in the case where $\Gamma$ has trivial valuation (i.e., at most one arrow between two vertices), if $F:k(\Gamma) \rightarrow \ind \Gamma$ is a $k$-linear functor such that $FX = X$ for every vertex $X$ of $\Gamma$ and $F(\overline{\alpha})$ is an irreducible morphism for every arrow $\alpha$ of $\Gamma$, then $F$ is a Riedtmann functor.
\end{obs}

We also need a key property of Riedtmann functors, that relates the filtrations of the radical of the mesh category with the ones of the module category.

\begin{teo}[consequence of \cite{CMT2}, Thm. B]
\label{th:b}
Let $F:k(\Gamma) \rightarrow \ind \Gamma$ be a Riedtmann functor. Then for every $n \geq 0$ and every pair of vertices $X,Y \in \Gamma_0$, the functor $F$ induces a bijection

$$ \frac{\mathcal{R}^nk(\Gamma)(X,Y)}{\mathcal{R}^{n+1}k(\Gamma)(X,Y)}\xrightarrow{\sim} \frac{\rad^n(X,Y)}{\rad^{n+1}(X,Y)}$$

\end{teo}

\section{Mesh-comparable components}
\label{sec:defi mesh-c}

In this section, we shall define mesh-comparability of components of the Auslander-Reiten quiver and state our first results on it.

\subsection{Definition of mesh-comparability}

\begin{defi}
Let $\Gamma$ be a component of the Auslander-Reiten quiver of an algebra $A$. We say that $\Gamma$ is \textbf{mesh-comparable} via $F: k(\Gamma) \rightarrow \ind \Gamma$ provided $F$ is a Riedtmann functor. In this case, we will also say that $F$ is a  \textbf{mesh-comparison}.
\end{defi}

Although the name `mesh-comparable' is being introduced here, there are existing results in literature, which construct Riedtmann functors using the well-known technique of {\it knitting}, and whose proofs can be easily adapted to give our first examples of mesh-comparable components. We highlight them in the following proposition:

\begin{prop}
\label{prop:easy exs mesh-c}

Suppose $\Gamma$ is an Auslander-Reiten component satisfying one of the following properties:

\begin{enumerate}
\item[(a)] $\Gamma$ {\it has length} (\cite{BG}), i.e., for every pair of vertices $x,y \in \Gamma_0$, all paths from $x$ to $y$ have the same length.

\item[(b)] $\Gamma$ is of type $\mathbb{Z} \Delta$, where $\Delta$ is a tree quiver.
\end{enumerate}

Then $\Gamma$ is mesh-comparable.

\end{prop}

\begin{proof}
    For the proof of (a), we construct the mesh-comparison $k(\Gamma) \rightarrow \ind \Gamma$ following the argument in \cite{BG}, 3.1b. The proof of (b) is done similarly, but this time following \cite{Rie}, 2.2.
\end{proof}

\subsection{Chosen irreducible morphisms}

As we are about show, existence of mesh-comparability is equivalent to being able to choose an irreducible morphism for each arrow of the Auslander-Reiten component, in such a way that irreducible morphisms corresponding to a mesh form almost split sequences.

\begin{defi}
Let $\Gamma$ be an Auslander-Reiten-component which is mesh-comparable via $F$. Given an arrow $\alpha: X \rightarrow Y$ of  $\Gamma$, we say that $F(\overline{\alpha}):X \rightarrow Y$ is a \textbf{chosen (irreducible) morphism}. In case we want to stress the functor $F$, we would say that it is an \textbf{$F$-chosen} morphism.
\end{defi}

Now we formally state the equivalence between mesh-comparability and the choice of irreducible morphisms.

\begin{prop}
\label{prop:mesh-c equiv escolha}
   The following statements are equivalent for an Auslander-Reiten component $\Gamma$:
   \begin{enumerate}
    \item[(a)] $\Gamma$ is mesh-comparable. 
    \item[(b)] For each arrow $\alpha: X \rightarrow Y$ of $\Gamma$, there exists an irreducible morphism $f_{\alpha}: X \rightarrow Y$ such that:

    \begin{enumerate}
        \item[(i)] If $X$ is a projective module of  $\Gamma$ and $\alpha_1:X_1 \rightarrow X, \ldots, \alpha_r:X_r \rightarrow X$ are all the arrows of $\Gamma$ ending in $X$, then  $[f_{\alpha_1} \cdots f_{\alpha_r}]: X_1 \oplus \cdots \oplus X_r \rightarrow X$ is a right almost split morphism. 
        
        \item[(ii)] If $X$ is an injective module of  $\Gamma$ and $\alpha_1:X \rightarrow X_1, \ldots, \alpha_r:X \rightarrow X_r$ are all the arrows of $\Gamma$ starting at $X$, then  $[f_{\alpha_1} \cdots f_{\alpha_r}]^T: X \rightarrow X_1 \oplus \cdots \oplus X_r$ is a left almost split morphism.

        \item[(iii)] If $X$ is a non-projective module in $\Gamma$ and

        
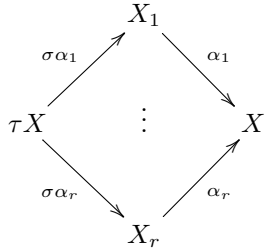
\begin{displaymath}
            \xymatrix{ & X_1 \ar[dr]^{\alpha_1} & \\
            \tau X \ar[ur]^{\sigma \alpha_1} \ar[dr]_{\sigma \alpha_r} & \vdots & X\\
            &X_r \ar[ur]_{\alpha_r} &}
        \end{displaymath}

        is the mesh ending at $X$, then

        $$0 \rightarrow \tau X \xrightarrow{[f_{\sigma \alpha_1} \cdots f_{\sigma \alpha_r}]^T} X_1 \oplus \ldots \oplus X_r \xrightarrow{[f_{\alpha_1} \cdots f_{\alpha_r}]} X \rightarrow 0$$

       is an almost split sequence. 
    \end{enumerate}
   \end{enumerate}
\end{prop}

\begin{proof}
    The implication  (a)$ \Rightarrow$ (b) is easy to see from the definitions. For the reverse implication, observe that the morphisms of the path category  $k\Ga$ are generated by arrows, and so  $\alpha \mapsto f_{\alpha}$ naturally induces a functor $F:k\Ga \rightarrow \ind \Ga$. The fact that the morphisms  $f_{\alpha}$ satisfy the mesh relations, as described in (iii), yields that $F$ induces a new functor $F: k(\Ga) \rightarrow \ind \Ga$. Then, (i), (ii) and (iii) guarantee that such a functor is indeed Riedtmann, showing (a).
\end{proof}

The following corollary is an easy consequence of the definition and Theorem~\ref{th:b}. It provides a first explanation of why mesh-comparability is interesting while studying the problem of compositions of irreducible morphisms.

\begin{cor}
\label{cor:composicao_bem_comportada_chosen}
  Let $\Gamma$ be mesh-comparable via $F$ and let $f_1: X_0 \rightarrow X_1, \cdots, f_n: X_{n-1} \rightarrow X_n$  be chosen irreducible morphisms between modules in $\Gamma$. Then either $f_n \cdots f_1 \in \rad^n (X_0,X_n) \setminus \rad^{n+1} (X_0,X_n)$ or $f_n \ldots f_1 = 0$.
\end{cor}

In another words, the composition of chosen irreducible morphisms is always well-behaved. 

\begin{proof} By definition of chosen irreducible morphisms, for every $i$ between 1 and $n$, there is an arrow $\alpha:X_{i-1} \rightarrow X_i$ such that $f_i = F(\overline{\alpha_i})$. Then $f_n \ldots f_1 = F(\overline{\alpha_n}) \ldots F(\overline{\alpha_1}) = F(\overline{\alpha_n} \ldots\overline{\alpha_1})$. If $f_n \ldots f_1 \in \rad^{n+1}(X_0,X_n)$, then, by Theorem~\ref{th:b}, we would have $\overline{\alpha_n} \ldots\overline{\alpha_1} \in \Rr^{n+1}(X_0,X_n)$, and since $k(\Gamma)$ is $\mathbb{N}$-graded, we will have actually that $\overline{\alpha_n} \ldots\overline{\alpha_1} = 0$, which implies $f_n \ldots f_1 = 0$.
\end{proof}

More generally, we will make use of the following result.

\begin{prop}
\label{prop:depth of F(phi)}
    Let $F:k(\Ga) \rightarrow \ind \Gamma$ be a mesh-comparison for an Auslander-Reiten component $\Gamma$. If $X,Y$ are modules of  $\Gamma$ and $\phi = \sum_{i=1}^t \lambda_1 \overline{\alpha_{i1}} \cdots \overline{\alpha_{in_i}} \in k(\Ga)(X,Y)$ is a non-zero morphism in the mesh category, where  $\lambda_i \in k$ and $\alpha_{ij}$ is an arrow of $\Ga$ for all $i,j$, then, denoting $n = \min\{n_1,\ldots,n_t\}$ and  $m = \max\{n_1,\ldots,n_t\}+1$, we have that $F(\phi) \in \rad^n(X,Y)$ and $F(\phi) \notin \rad^m(X,Y)$.
\end{prop}

\begin{proof}
    To verify that $F(\phi) \in \rad^n(X,Y)$ is a simple task. Let us show that $F(\phi) \notin \rad^m(X,Y)$. By Theorem~\ref{th:b}, $F$ induces a $k$-linear bijection

$$\frac{\Rr^{m-1} k(\Ga)(X,Y)}{\Rr^m k(\Ga)(X,Y)}\rightarrow \frac{\rad^{m-1}(X,Y)}{\rad^m (X,Y)}$$ 

    Hence, if  $F(\phi) \in \rad^m(X,Y)$, then $\phi \in \Rr^m k(\Ga)(X,Y)$. Since $k(\Ga)$ is an  $\mathbb{N}$-graded category and $\phi = \sum_{i=1}^t \lambda_1 \overline{\alpha_{i1}} \cdots \overline{\alpha_{in_i}} \in k(\Ga)(X,Y)$ is a linear combination of morphisms of degrees strictly smaller than  $m$, the fact that $\phi \in \Rr^m k(\Ga)(X,Y)$ should imply that $\phi = 0$, contrary to the hypothesis of $\phi$ being nonzero. This completes the proof.
\end{proof}

\subsection{First connection with standardness} Recall that an Auslander-Reiten component $\Gamma$ is called {\bf standard} provided there exists an isomorphism of categories $i:k(\Ga) \xrightarrow{\sim} \ind \Gamma$.

\begin{prop}
\label{prop:standard eh mesh-c}
If $\Gamma$ is standard, then there is an isomorphism of categories $F:k(\Gamma) \xrightarrow{\sim} \ind \Gamma$ which is also a Riedtmann functor (i.e, a mesh-comparison). In particular, any standard Auslander-Reiten component is mesh-comparable. 
\end{prop}

\begin{proof}
Let $i$ be an isomorphism of categories. Then $i$ is not necessarily the identity on objects and so, it is not necessarily a Riedtmann functor. By considering a suitable automorphism $G: k(\Gamma) \rightarrow k(\Gamma)$, we construct the composition $F = i \circ G: k(\Gamma) \rightarrow \ind \Gamma$, which is an isomorphism of categories and also a Riedtmann functor. 
\end{proof}

We have given our first examples of mesh-comparable components with Proposition~\ref{prop:easy exs mesh-c}, and the result above adds more examples to that list: since finite directed components, preprojective and postprojective components, and connecting components of tilted algebras are all standard (see, e.g., \cite{Rin}), we immediately obtain that all these are mesh-comparable too.

The next two results will give a better understanding of how close a mesh-comparison is from being an isomorphism (what would imply standardness).

\begin{prop}
\label{prop:fiel e denso}
   Any Riedtmann functor $F: k(\Gamma) \rightarrow \ind \Gamma$ is faithful and dense.             
\end{prop}

\begin{proof}
    Since $F$ works as the identity on the objects, it is clearly dense. It remains to show that it is faithful. For that, let $\phi \in k(\Ga)(X,Y)$ be a morphism in the mesh category between $X$ and $Y$ such that $F(\phi) = 0$. Because of the Proposition~\ref{prop:depth of F(phi)}, either there is an $m$ such that $F(\phi) \notin \rad^m(X,Y)$, or $\phi = 0$. Since $F(\phi) = 0$, the first case cannot happen, showing that $\phi = 0$ and so that $F$ is faithful.
\end{proof}

\begin{cor}
\label{cor:F iso F pleno}
    Let  $F: k(\Gamma) \rightarrow \ind \Gamma$ be a Riedtmann functor. The following statements are equivalent:

    \begin{enumerate}
        \item[(a)] $F$ is an isomorphism (and hence  $\Gamma$ is standard)
        \item[(b)] each radical morphism between modules of $\Gamma$ is a linear combination of compositions of $F$-chosen morphisms. 
        \item[(c)] $F$ is full.
    \end{enumerate}
\end{cor}

\begin{proof}
  (a) $\Rightarrow $ (b) follows easily from the definitions. 
    
  (b) $ \Rightarrow $ (c):  Let  $f$ be a morphism in $ \Hom(X,Y)$, with $X$ and $Y$ belonging to $\Gamma$. If $f$ is radical, then $f = \sum_{i=1}^t \lambda_i F(\overline{\alpha_{i1}})\cdots F(\overline{\alpha_{in_i}})$, where each $\alpha_{ij}$ is an arrow of $\Gamma$. Then it follows that $f =  F \left( \sum_{i=1}^t \lambda_i \overline{\alpha_{i1}}\cdots \overline{\alpha_{in_i}}\right)$, that is, $f$ is the image of a morphism via $F$.

    If $f$ is not radical, then it is an isomorphism $X \cong Y$ (because both $X$ and $Y$ are indecomposable), and since $k$ is algebraically closed, there exists  $\lambda \in k$ such that $f - \lambda. 1_X \in \rad(X,X)$. Therefore $f - F(\lambda.1_X) \in \rad(X,X)$ and using the first part we conclude that $f$  is the image of a morphism via $F$. 
    
    In order to show (c) $ \Rightarrow $ (a), it is enough to apply Proposition~\ref{prop:fiel e denso} and the fact that $F$ acts as the identity on the objects. 
\end{proof}

\section{Decomposing a morphism in its parts}
\label{sec:decomposition}

By the existence of a mesh-comparison, we can establish a decomposition of any morphism between indecomposable modules, in a manner which we now explain. From that decomposition, we can study interesting properties of the morphisms, such as depth or being (finitely) generated by chosen morphisms.

\subsection{The decomposition}

We need a preparatory lemma. 

\begin{lem}
\label{lem:almost infty}
Let  $\Gamma$ be an Auslander-Reiten component and assume that it is mesh-comparable via $F$. If  $f:X \rightarrow Y$ is an irreducible morphism between modules of $\Gamma$, then there exists a morphism $\phi \in k(\Gamma)(X,Y)$ such that $f - F(\phi) \in \rad^{\infty}(X,Y)$.
\end{lem}

\begin{proof}
   Let  $m \in \mathbb{N}$ be such that $\rad^m(X,Y) = \rad^{\infty}(X,Y)$. Because of Theorem~\ref{th:b}, we know that the following map induced by $F$
$$\frac{\Rr^l k(\Ga)(X,Y)}{\Rr^{l+1}k(\Ga)(X,Y)}\rightarrow \frac{\rad^l(X,Y)}{\rad^{l+1}(X,Y)}$$
is surjective for all $l \geq 0$. Hence, there exists  $\phi_0 \in k(\Ga)(X,Y)$ such that $f-F(\phi_0) \in \rad(X,Y)$. In the same way, there exists $\phi_1 \in \Rr k(\Ga)(X,Y)$ such that $f-F(\phi_0)-F(\phi_1) \in \rad^2(X,Y)$, and so recursively, we build for each $0 \leq l < m$ morphisms $\phi_l \in \Rr^l k(\Ga)(X,Y)$ such that  $f-F(\phi_0)-F(\phi_1)-\cdots - F(\phi_{m-1}) \in \rad^m(X,Y) = \rad^{\infty}(X,Y)$. Therefore, it suffices to take  $\phi = \phi_0 + \cdots + \phi_{m-1}$ to conclude the proof. 
\end{proof}

Now we establish the promised decomposition of a morphism in its parts.

\begin{prop}
\label{prop:decomposicao}
Let $\Gamma$ be an Auslander-Reiten component, mesh-comparable via $F$, and consider a morphism $f:X \rightarrow Y$ between modules of $\Gamma$. Then there exists a decomposition (with finite support) 
$$f = f_0 + f_1 + \cdots + f_n + \cdots + f_{\infty}$$
satisfying the following:
\begin{enumerate}
    \item[(i)] $f_0$ is either zero or an isomorphism;
    \item[(ii)] for each finite $n>0$, the morphism $f_n$ is either zero or a linear combination of morphisms of the form  $F(\overline{\alpha_n})\cdots F(\overline{\alpha_1})$, where $X = X_0 \xrightarrow{\alpha_1} X_1 \xrightarrow{\alpha_2} \cdots \xrightarrow{\alpha_n} X_n = Y$ is a path in $\Gamma$; and
    \item[(iii)] $f_{\infty} \in \rad^{\infty}(X,Y)$.
\end{enumerate}
Moreover, such decomposition is unique, that is, if $f = g_0 + g_1 + \cdots + g_{\infty}$ is a decomposition with the same properties of the first, then $f_n = g_n$ for all $0 \leq n \leq \infty$.
\end{prop}

\begin{proof}
    Because of  Lemma~\ref{lem:almost infty}, there exists  $\phi \in k(\Ga)(X,Y)$ such that $f-F(\phi) \in \rad^{\infty}(X,Y)$. So we take $f_{\infty} = f-F(\phi)$. Furthermore, by the  $k$-linearity of $F$, $F(\phi)$ can be decomposed as a sum with finite support $f_0 +f_1+ \cdots$, where each $f_n$ is as in the statement. This completes the proof of the existence. 
   
Let us prove the uniqueness. If  $f = f_0 + f_1 + \cdots + f_{\infty} = g_0 + g_1 + \cdots +g_{\infty}$ are as in the statement, then  $f_0 -g_0 \in \rad(X,Y)$. Since $X$ and $Y$ are indecomposable, $\rad(X,Y)$ has codimension either 0 or 1 in $\Hom(X,Y)$, and from that we conclude that $f_0 = g_0$.

Furthermore, we have  $f_1 + f_2 + \cdots + f_{\infty} = g_1 + g_2 + \cdots +g_{\infty}$, and so $f_1 -g_1 \in \rad^2(X,Y)$. Write $f_1 = \lambda_1 F(\overline{\alpha_1}) + \cdots + \lambda_m F(\overline{\alpha_m})$ and $g_1 = \mu_1 F(\overline{\beta_1}) + \cdots + \mu_l F(\overline{\beta_l})$, where the $\lambda_i$'s, $\mu_i$'s are scalars and the $\alpha_i$'s, $\beta_i$'s are arrows. Then 
$$\lambda_1 F(\overline{\alpha_1}) + \cdots + \lambda_m F(\overline{\alpha_m}) - \mu_1 F(\overline{\beta_1}) - \cdots - \mu_l F(\overline{\beta_l}) = F(\lambda_1 \overline{\alpha_1} + \cdots + \lambda_m \overline{\alpha_m} - \mu_1 \overline{\beta_1} - \cdots - \mu_l \overline{\beta_l})$$ 
lies in $\rad^2(X,Y). $
Because of the injectivity of the maps induced by $F$ (Theorem~\ref{th:b}), it follows that  
$$\lambda_1 \overline{\alpha_1} + \cdots + \lambda_m \overline{\alpha_m} - \mu_1 \overline{\beta_1} - \cdots - \mu_l \overline{\beta_l}\ \  \in\ \ \mathcal{R}^2 k(\Ga)(X,Y),$$ 
and, since this is a linear combination of the homogeneous elements of degree one in the  $\mathbb{N}$-graded category  $k(\Ga)$, we must have
$$\lambda_1 \overline{\alpha_1} + \cdots + \lambda_m \overline{\alpha_m} - \mu_1 \overline{\beta_1} - \cdots - \mu_l \overline{\beta_l} =0. $$ 
Which yields $f_1 = g_1$.
    
Using similar arguments, one can show that, for each $n$, $f_n = g_n$. Since the sums $f_0 + f_1 + \cdots + f_{\infty} $ and  $ g_0 + g_1 + \cdots +g_{\infty}$ have finite support, we also obtain that  $f_{\infty} = g_{\infty}$, concluding the proof. 
\end{proof}

\begin{obs}
It is important to emphasize that the uniqueness of the decomposition stated in Proposition~\ref{prop:decomposicao} is on the morphisms $f_n$ and not on the representation of each of them as linear combinations of images under $F$ of paths in $\Gamma$. By the way, such uniqueness will not hold in general due to the mesh relations, as we can see from the next example.
\end{obs}

\begin{ex}
   Let $A$ be the path algebra given by the quiver $1 \leftarrow 2 \leftarrow 3$. Its Auslander-Reiten quiver is 

    
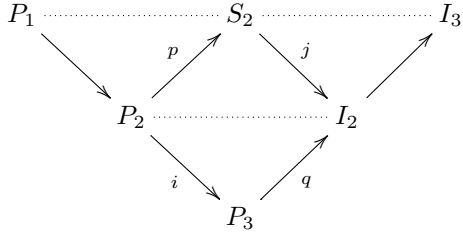
\begin{displaymath}
        \xymatrix{P_1 \ar[dr] \ar@{.}[rr]&& S_2 \ar[dr]^{j} \ar@{.}[rr]&& I_3 \\
        & P_2 \ar[dr]_{i} \ar[ur]^{p}\ar@{.}[rr]&& I_2 \ar[ur]& \\
        && P_3 \ar[ur]_{q} &&}
    \end{displaymath}
where $P_l, I_l$ and $S_l$, respectively, stand for the indecomposable projective, indecomposable injective and simple module associated with a vertex $l = 1,2 \text{ or } 3$, the maps $i,j$ are canonical inclusions and $p,q$ are canonical projections. Observe that, by choosing the correct irreducible morphisms, it is possible to construct a mesh-comparison over $\Gamma$(mod$A$) in such a way that $i,j,p$ and $-q$ are chosen irreducible morphisms. 

Now define  $f \doteq jp = -(-q)i$. Then  $f$ can be written in two different ways as a linear combination of compositions of chosen morphisms. This, however, does not contradict the uniqueness stated in Proposition~\ref{prop:decomposicao}: as we have observed above, $jp$ and $-(-q)i$ are the same morphism, which is the morphism denoted by $f_2$ in the decomposition above and which is uniquely determined from $f$.
\end{ex}

\subsection{Parts of a morphism}

\begin{defi}
\label{def:partes}
Given a morphism $f: X \rightarrow Y$, using the notation given in Proposition~\ref{prop:decomposicao}, we define the following morphisms:

\begin{itemize}
    \item For each $i \in \mathbb{N}$, define $\ff_i(f) = f_i$, which is the {\bf i-th part} of $f$.
    
    \item Define $\ff_{\infty}(f) \doteq f_{\infty}$, which is the {\bf infinite part} of $f$. 
    
    \item The {\bf finite part} of $f$ is  $\ff_{< \infty}(f) = f - \ff_{\infty}(f)$. 
    
    \item In case $f \neq 0$, let $i = 0,1,\ldots,\infty$ be the smallest index such that $\ff_i (f) \neq 0$. We then define $\ff_P(f) \doteq \ff_i (f)$ as the {\bf principal part of} $f$  and $\ff_S(f) \doteq f - \ff_i (f)$ as the {\bf secondary part of} $f$.
\end{itemize}
\end{defi}

This definition indicates ways to decompose a morphism $f$ using Proposition~\ref{prop:decomposicao}: 

\begin{align*}
    f &= \ff_0(f)+\ff_1(f)+ \ldots + \ff_{\infty}(f) \\
    &= \ff_{<\infty}(f) + \ff_{\infty}(f) = \ff_P(f)+\ff_S(f).
\end{align*}

\begin{ex}
\label{ex:exemplo 0 parte 2}
Let $A$ be the $k$-algebra given by the quiver 

    
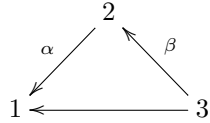
\begin{displaymath}
    \xymatrix{& 2 \ar[dl]_{\alpha}& \\
    1 & & 3 \ar[ul]_{\beta} \ar[ll]}
    \end{displaymath}
 bounded by $\alpha \beta = 0$. 
Its Auslander-Reiten quiver is 
    
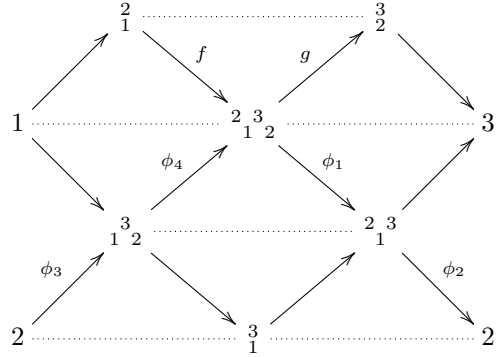
\begin{displaymath}
        \xymatrix{ 
        &{\begin{smallmatrix} 2 \\ 1 \end{smallmatrix}} \ar[dr]^f \ar@{.}[rr] && {\begin{smallmatrix} 3 \\ 2 \end{smallmatrix}} \ar[dr]& \\
       1 \ar[dr] \ar[ur] \ar@{.}[rr] && {\begin{smallmatrix} 2 \hspace{1ex} 3 \hspace{1ex} \\ \hspace{1ex} 1 \hspace{1ex} 2 \end{smallmatrix}} \ar@{.}[rr] \ar[dr]^{\phi_1} \ar[ur]^g && 3\\
        & {\begin{smallmatrix} 3 \\ 1 \hspace{1ex} 2 \end{smallmatrix}} \ar@{.}[rr] \ar[dr] \ar[ur]^{\phi_4} && {\begin{smallmatrix} 2 \hspace{1ex} 3 \\ 1 \end{smallmatrix}} \ar[dr]^{\phi_2} \ar[ur] &\\
        2 \ar[ur]^{\phi_3} \ar@{.}[rr] && {\begin{smallmatrix} 3 \\ 1 \end{smallmatrix}} \ar@{.}[rr] \ar[ur] && 2}
    \end{displaymath}
where we denote each module by its composition series, and $f,g,\phi_1,\phi_2,\phi_3,\phi_4$ denote the natural inclusions and projections between the corresponding modules. Observe that such morphisms are irreducible.

Note that this single component of the Auslander-Reiten quiver is mesh-comparable: in this case one can define a choice of irreducible morphisms which satisfy the mesh relations using the knitting technique (see \cite{AC}). Moreover, one can make this choice in such a way that the morphisms $f,g,\phi_1,\phi_2,\phi_3,\phi_4$ are chosen irreducible. Since $g$ in particular is chosen irreducible, we already get that $\ff_0(g) = 0$, $\ff_1(g) = g$ and $\ff_i(g) = 0$ for $1 < i \leq \infty$.

Now consider $f' = f+\phi_4\phi_3\phi_2\phi_1 f$. Then, $f'$ is also an irreducible morphism, but it is not a chosen one. In this case, we have $\ff_0(f') = 0$, $\ff_1(f') = f$, $\ff_2(f') = \ff_3(f') = \ff_4(f') = 0$, $\ff_5(f') = \phi_4\phi_3\phi_2\phi_1 f$ and $\ff_i(f') = 0$ for $5 < i \leq \infty$. The principal and the secondary parts of the morphisms $g$ and $f'$ are the following: 
    \begin{align*}
        \ff_P(g) = g & \hspace{1cm} \ff_P(f') = f \\
        \ff_S(g) = 0 & \hspace{1cm} \ff_S(f') = \phi_4\phi_3\phi_2\phi_1 f
    \end{align*}
\end{ex}

Our next result characterizes when a morphism lies in rad$^{\infty}$ in terms of its parts. 

\begin{prop} 
\label{prop:calculo_infty}
For a morphism $f: X \rightarrow Y$ between modules in a mesh-comparable component, the following are equivalent:
 \begin{enumerate}
     \item $f$ belongs to $ \rad^{\infty}(X,Y)$; \item $f = \ff_{\infty}(f)$; 
     \item $\ff_{<\infty}(f) = 0$.
    \end{enumerate}
\end{prop}

\begin{proof} 
The implication $(1) \Rightarrow (2)$ follows applying the uniqueness stated in  Proposition~\ref{prop:decomposicao} on the decomposition $f = 0 + f \in \rad^{\infty}(X,Y)$. The implications $(2) \Rightarrow (3)$ and $(3)\Rightarrow (1)$ are immediate. 
 
\end{proof}

Also, the operation of taking finite parts behaves well under compositions: 

\begin{prop} 
Let $f: X \rightarrow Y$ and  $g: Y \rightarrow Z$ be morphisms between modules in a mesh-comparable component. Then $\ff_{<\infty}(gf) = \ff_{<\infty}(g)\ff_{<\infty}(f)$.
\end{prop}

\begin{proof}  Observe that         
        \begin{align*}
            gf &= (\ff_{<\infty}(g)+\ff_{\infty}(g))(\ff_{<\infty}(f)+\ff_{\infty}(f)) \\
            &= \ff_{<\infty}(g)\ff_{<\infty}(f) + \ff_{\infty}(g)\ff_{<\infty}(f) + \ff_{<\infty}(g)\ff_{\infty}(f) + \ff_{\infty}(g)\ff_{\infty}(f)
        \end{align*}
Since the last three terms in the sum above belong to $\rad^{\infty}$, we get that the first term is the finite part of $gf$, as required. 
\end{proof}

\subsection{Generalized standard components}

Recall that an Auslander-Reiten component $\Ga$ is called {\bf generalized standard} \cite{Sko} provided $\rad^{\infty}(X,Y) = 0$ for all $X, Y \in \Ga$. The next result is an immediate consequence of Proposition~\ref{prop:calculo_infty}. Interestingly, it states that for a component, being generalized standard is equivalent to the vanishing of the infinite parts of morphisms between its modules.

\begin{prop}
\label{prop:parte inf e gen st}
Assume that a component $\Gamma$ is mesh-comparable via $F$. Then, $\Gamma$ is generalized standard if and only if $\ff_{\infty}(f) = 0$ for all morphisms $f$ between modules of  $\Gamma$, which is also equivalent to say that $f = \ff_{<\infty}(f)$ for each morphism $f$.
\end{prop}

\subsection{The depth of a morphism}

Let $f: X \rightarrow Y$ be a morphism in $\md A$. The \textbf{depth} of $f$, denoted by $\Dp(f)$, is the smaller natural number $n \geq 0$ such that  $f \in \rad^n(X,Y)$ and $f \notin \rad^{n+1}(X,Y)$, if such number exists. Otherwise we define $\Dp(f) = \infty$.

\begin{obs}
Let $f$ be a morphism and $0 \leq i \leq \infty$. By construction, $\ff_i(f) \in \rad^i$. If $i$ is finite  and $\ff_i(f) \neq 0$, then, applying Proposition~\ref{prop:depth of F(phi)}, we have that $\ff_i(f) \notin \rad^{i+1}$. This means that for every $0 \leq i \leq \infty$, $\Dp(\ff_i(f)) = i$, i.e, the $i$-part of $f$ has depth $i$, provided it is non-zero.
\end{obs}

As we are about to see, the depth of a non-zero morphism coincides with the index of its principal part.

\begin{prop}
If $f$ is a nonzero morphism between modules in a mesh-comparable component and $n \in \{0,1,\ldots,\infty\}$, then $\Dp(f) = n$ if and only if $\ff_i(f) = 0$ for all $i<n$  and $\ff_n(f) \neq 0$.
\end{prop}

\begin{proof}
 {\it Necessity.}
Suppose that  $\Dp(f) = n$. First we verify that $\ff_i(f) = 0$ for all $i<n$. If $n=0$, then this is true by emptiness. If now $n > 0$, then $f$ is radical, and, given that $f = \ff_0(f) + \ff_1(f) + \cdots + \ff_{\infty}(f)$, this means $\ff_0(f) = f-\ff_1 (f) - \cdots - \ff_{\infty}(f)$ is also radical, and so $\ff_0(f) = 0$ by the above remark. By a similar argument, one can show that $\ff_1(f) = \cdots = \ff_{n-1}(f) = 0$. Now we verify that $\ff_n(f) \neq 0$. From the first part we have $f = \ff_n(f) + \ff_{n+1}(f) + \cdots + \ff_{\infty}(f)$, with $f \notin \rad^{n+1}$ by hypothesis and the fact that $\ff_{n+1}(f) + \cdots + \ff_{\infty}(f) \in \rad^{n+1}$. This shows that $\ff_n(f) \neq 0$.
  
{\it Sufficiency.} 
Suppose that $\ff_i(f) = 0$ for all  $i<n$ and that  $\ff_n(f) \neq 0$. If $n= \infty$, then $f = \ff_{\infty}(f) \in \rad^{\infty}$ as required. So suppose $n < \infty$. Then $f = \ff_n(f)+ \ff_{n+1}(f) + \cdots + \ff_{\infty}(f)$, with $\ff_n(f) \in \rad^n \setminus \rad^{n+1}$ (by the remark above) and with $\ff_{n+1}(f) + \cdots + \ff_{\infty}(f) \in \rad^{n+1}$. This shows that  $f \in \rad^n \setminus \rad^{n+1}$ and the proof is finished.
\end{proof}

\subsection{Paths of chosen irreducible morphisms}

By the next two results, we establish a curious connection between the radical filtration of the space of morphisms between two modules and an upper bound on the length of nonzero paths of chosen morphisms between such modules.

\begin{prop}
\label{prop: upper bound}
    Let $\Gamma$ be an Auslander-Reiten component and assume it is mesh-comparable via $F$. For each $m \geq 1$ and all modules $X, Y$ of $\Gamma$, the following statements are equivalent:
    \begin{enumerate}
        \item[(a)] $\rad^m(X,Y) = \rad^{\infty}(X,Y)$.
        \item[(b)] each path of chosen irreducible morphisms between  $X$ and $Y$ with nonzero composition has length strictly smaller than $m$.
    \end{enumerate}
\end{prop}

\begin{proof}
    (a) $\Rightarrow$ (b). Let $X \xrightarrow{f_1} \cdots \xrightarrow{f_n} Y$ be a path of chosen irreducible morphisms with nonzero composition. We shall prove that $n < m$. Since $f_n \cdots f_1 \neq 0$, it follows from Corollary~\ref{cor:composicao_bem_comportada_chosen} that $f_n \ldots f_1 \in \rad^n(X,Y) \setminus \rad^{n+1}(X,Y)$. In particular, $f_n \ldots f_1 \notin \rad^{\infty}(X,Y) = \rad^m(X,Y)$ and so $n < m$ as required. 

    (b)$ \Rightarrow$ (a). Suppose there exists a morphism $f \in \rad^m(X,Y) \setminus \rad^{\infty}(X,Y)$. So, there exists an  $n \geq m$ such that $\Dp(f) = n < \infty$, which implies that $\ff_n(f) \neq 0$. By definition,  $\ff_n(f)$ is a linear combination of compositions of at least $n$ chosen morphisms, and then we can conclude that there exists a path of chosen morphisms of length greater than or equal to $n$ with nonzero composition. By hypothesis, $n < m$, a contradiction to how $n$ was defined. Hence  $\rad^m(X,Y) = \rad^{\infty}(X,Y)$.
\end{proof}

\begin{cor}
\label{cor:upper bound}
    Let $\Gamma$ be mesh-comparable via $F$ and let $X,Y$ be modules in $\Gamma$. Let  $m$ be the smaller natural number such that $\rad^m(X,Y) = \rad^{\infty}(X,Y)$. Then:

    \begin{enumerate}
        \item[(a)] If $m \geq 1$, then there exists a nonzero path of chosen irreducible morphisms between $X$ and $Y$ of length $m-1$.

        \item[(b)] The number $m$ is strictly larger than the length of any nonzero path of chosen irreducible morphisms between $X$ and $Y$. In particular, if  $m=0$, then there is no such path.
    \end{enumerate}
\end{cor}

\begin{proof} (a) By definition, there exists a morphism  $f \in \rad^{m-1}(X,Y) \setminus \rad^m (X,Y)$. Therefore $\Dp(f) = m-1$ and so $\ff_{m-1} (f) \neq 0$. Hence there exists a nonzero path of chosen morphisms between $X$ and $Y$ of length $m-1$.\\
(b) It follows directly from Proposition~\ref{prop: upper bound}.
\end{proof}

\section{A new proof that standard implies generalized standard}
\label{sec:new proof Liu}

The results we have given so far allow us to give two new proofs of the essential result by S. Liu (\cite{Liu3}) which states that each standard component is generalized standard. We consider these proofs significantly simpler than the original one. 

\begin{teo}[\cite{Liu3}]
    Let $\Gamma$ be a standard Auslander-Reiten component of an algebra $A$. Then, for any pair of modules $X,Y$ of $\Gamma$, it holds that $\rad^{\infty}_A(X,Y) = 0$, that is, $\Gamma$ is generalized standard.
\end{teo}

\begin{proof}
The initial part of the argument is the same from the proof in the original paper (\cite{Liu3}, Main Theorem), which we adapt here: 
since $\Gamma$ is standard, it is mesh-comparable, and we can also assume due to standardness that each morphism between modules in $\Gamma$ is a $k$-linear combination of compositions of chosen irreducible morphisms. Suppose that $\rad^{\infty}(X_0,Y_0) \neq 0$ for a pair of modules $X_0,Y_0$ of $\Gamma$. Then there exist nonzero paths of arbitrary length of chosen irreducible morphisms from  $X_0$ to  $Y_0$.

From this point, we can get a contradiction in an easier way than in the original proof, not depending of the preparatory lemmata established there. It goes as follows. 
By Corollary~\ref{cor:upper bound}, if $m \in \mathbb{N}$ is such that  $\rad^m(X_0,Y_0) = \rad^{\infty}(X_0,Y_0)$, then  $m$ is an upper bound on the length of the nonzero paths of chosen morphisms between $X_0$ and $Y_0$. This contradicts what was said above, and so we complete the proof. 
\end{proof}

The second proof has a theoretical path even shorter than the one given above.

\begin{proof}[Second proof.]
Since $\Gamma$ is standard, by Proposition~\ref{prop:standard eh mesh-c} there exists a functor $F:k(\Gamma) \rightarrow \ind \Gamma$ which is an isomorphism and a mesh-comparison. Let $f \in \rad^{\infty}(X,Y)$, with $f \neq 0$. Since $F$ is full (Corollary~\ref{cor:F iso F pleno}), there exists a nonzero $\phi \in k(\Gamma)(X,Y)$ such that $f = F(\phi)$. Because of Proposition~\ref{prop:depth of F(phi)}, there exists $m > 0$ such that $F(\phi) \notin \rad^m(X,Y)$, a contradiction to the fact that $f \in \rad^{\infty}(X,Y)$. This shows that  $\Gamma$ is generalized standard.
\end{proof}

\section{Mesh-comparability and the problem of composing morphisms}
\label{sec:mesh-c x compositions}

Since compositions of chosen irreducible morphisms are well-behaved (Corollary~\ref{cor:composicao_bem_comportada_chosen}), we can expect mesh-comparable components to be suitable for the study of the problem of composition of irreducible morphisms. 

Indeed, a general criterion for this problem was given in \cite{CMT1}, Proposition 5.1, and we can rephrase this criterion using the terminology introduced for mesh-comparable components. We believe this new version can be more useful in practical examples. Also note that, differently from \cite{CMT1}, we do not have to impose here that the morphisms $f_i$'s below are irreducible.

\begin{teo} (compare with \cite{CMT1}, Proposition 5.1)
\label{th:composta em rad n+1}
Let $f_1:X_0 \rightarrow X_1, \cdots, f_n:X_{n-1} \rightarrow X_n$ be morphisms in a mesh-comparable Auslander-Reiten component $\Gamma$ and denote  $N = \Dp(f_1) + \ldots + \Dp(f_n)$. Then  $f_n \cdots f_2 f_1 \in \rad^{N+1}(X_0,X_n)$ if and only if $$\ff_P(f_n)\cdots \ff_P(f_2)\ff_P(f_1) = 0.$$
\end{teo}

\begin{proof}
The proof can be implicitly taken from \cite{CMT1}, Prop. 5.1, but using our notation here, we shall repeat the argument there in a simplified way. \\
Since, for each $i$, $f_i = \ff_P(f_i) + \ff_S(f_i)$, we have that 

    $$f_r \ldots f_1  = \sum_{J_1 \in \{P,S\}} \ldots \sum_{J_r \in \{P,S\}} \ff_{J_r}(f_r) \ldots \ff_{J_1}(f_1).$$

If, in any term of the above sum, we have that $J_i$ equals $S$, then $\ff_{J_r}(f_r) \ldots \ff_{J_1}(f_1) $ belongs to $ \rad^{N+1}(X_0,X_n)$. On the other hand, as a consequence of Theorem~\ref{th:b} or Proposition~\ref{prop:depth of F(phi)}, $\ff_P(f_r) \ldots \ff_P(f_1)$ is either zero or does not belong to  $\rad^{N+1}(X_0,X_n)$. The result follows.
\end{proof}

In particular, the above result exchanges the decision on whether the composition of morphisms belongs to a higher power of the radical to a decision on whether the composition of their principal parts is zero, and this latter is usually a simpler task than the former, as the next example shows:

\begin{ex}
Let us return to the Example~\ref{ex:exemplo 0 parte 2}. There, we have calculated the principal parts of the irreducible morphisms $g$ and $f'$. Since we have that  $\ff_P(g)\ff_P(f') = gf = 0$, we can conclude, from Theorem~\ref{th:composta em rad n+1}, that $gf' \in \rad^3$. 
Indeed, the same example is carried in \cite{CCT1}, and there the authors conclude that $gf' \in \rad^3$ by a naive argument. 
\end{ex}

Our next result is an extension of Theorem 5.1.3 from \cite{CCsurvey} to mesh-comparable components. As also explained in \cite{CCsurvey}, it will be an application of a number of other results, namely Theorem 2.7 and Corollary 2.9 from \cite{CT}, Proposition 5.1 from \cite{CMT1} and Proposition 3 from \cite{CMT2}. 

\begin{teo}
\label{th:inf ou atalhos}

Let $\Ga$ be a mesh-comparable component of the Auslander-Reiten quiver of an algebra $A$.
Given indecomposable modules $X_0,X_1,\ldots,X_n$ in $\Ga$, the following are equivalent:

\begin{enumerate}
    \item There is a path of irreducible morphisms $X_0 \xrightarrow{h_1} X_1 \xrightarrow{h_2} \ldots \xrightarrow{h_n} X_n$ such that $h_n \ldots h_1$ is non-zero and belongs to $\rad^{n+1}(X_0,X_n)$.

    \item There is a path of irreducible morphisms $X_0 \xrightarrow{f_1} X_1 \xrightarrow{f_2} \ldots \xrightarrow{f_n} X_n$ such that $f_n \ldots f_1 =0$, each $f_i$ being a linear combination of chosen irreducible morphisms, and there are morphisms $X_0 \xrightarrow{\epsilon_1} X_1 \xrightarrow{\epsilon_2} \ldots \xrightarrow{\epsilon_n} X_n$ with $\epsilon_n \ldots \epsilon_1 \neq 0$ and satisfying that, for every $1 \leq i \leq n$, either $\epsilon_i = f_i$ or $\epsilon_i \in \rad^2(X_{i-1},X_i)$ is a linear combination of compositions of at least two chosen irreducible morphisms.

    \item There is a path of irreducible morphisms $X_0 \xrightarrow{f_1} X_1 \xrightarrow{f_2} \ldots \xrightarrow{f_n} X_n$ such that $f_n \ldots f_1 =0$, each $f_i$ being a linear combination of chosen irreducible morphisms, and such that one of the following holds:
    \begin{enumerate}
        
        \item  There is a path of irreducible morphisms $X_0 \xrightarrow{h_1} X_1 \xrightarrow{h_2} \ldots \xrightarrow{h_n} X_n$ such that $h_n \ldots h_1 \in \rad^{\infty}(X_0,X_n)\setminus\{0\}$.
    
        \item There are indices $1 \leq i_1 < \ldots < i_l \leq n$ such that for every $1 \leq j \leq l$, there is a linear combination of compositions of at least two chosen irreducible morphisms $\phi_{i_j}:X_{i_j-1} \rightarrow X_{i_j}$ such that

    $$f_n \ldots f_{i_l+1} \phi_{i_l} f_{i_l-1} \ldots f_{i_1+1} \phi_{i_1} f_{i_1-1} \ldots f_1 \neq 0.$$
    \end{enumerate}
\end{enumerate}

\end{teo}

\begin{proof}

The result follows from the proof made in \cite{CCsurvey}, Theorem 5.1.3. There we considered a Riedtmann functor $F:k(\tilde{\Gamma})\rightarrow \ind \Gamma$, where $\tilde{\Gamma}$ is the universal covering of $\Gamma$, but the same argument is valid if we replace this functor by a mesh-comparison $F:k(\Gamma) \rightarrow \ind \Gamma$. By the construction there, each $f_i$, $\epsilon_i$ or $\phi_i$ is the image via $F$ of a morphism in $k(\Gamma)$.
\end{proof}

\section{Standardness x mesh-comparability: the key result}
\label{sec:standard x mesh-c}

We have already shown with Proposition~\ref{prop:standard eh mesh-c} that every standard component is mesh-comparable. Our aim in this section is to give a key result which relates the concepts of standard, generalized standard and mesh-comparable components, that will show in particular that there are mesh-comparable components which are not standard.

\begin{lem}
\label{lem:gen st + mesh-comp}
    Let $\Gamma$ be a component of $\Gamma(\md A)$ and let $F:k(\Gamma) \rightarrow \ind \Gamma$ be a Riedtmann functor. Then the following functor induced by  $F$ (that is, the composition of $F$ with the projection $\ind \Gamma \rightarrow \ind \Gamma / \rad^{\infty}$)

    $$\overline{F}: k(\Gamma) \rightarrow \ind \Gamma / \rad^{\infty}$$
is an equivalence of categories. 
\end{lem}

\begin{proof}
     Observe that the same proof made to show that $F$ is faithful and dense in Proposition~\ref{prop:fiel e denso} also works to show that  $\overline{F}$ is faithful and dense. It remains to show that $\overline{F}$ is full.

    Let $f: X \rightarrow Y$ be a morphism between modules of  $\Gamma$. By  Lemma~\ref{lem:almost infty}, there exists a morphism  $\phi \in k(\Gamma)(X,Y)$ such that $f - F(\phi) \in \rad^{\infty}(X,Y)$, that means, $\overline{F}(\phi) = f + \rad^{\infty}(X,Y)$, and so $\overline{F}$ is full.
\end{proof}

\begin{teo}
\label{th:gen st + mesh-comp}
   Any component of the Auslander-Reiten quiver which is both generalized standard and mesh-comparable is also standard.
\end{teo}

\begin{proof}
    Because $\rad^{\infty} = 0$ in $\Gamma$, we get from Lemma~\ref{lem:gen st + mesh-comp} that 
    $$k(\Gamma) \cong \ind \Gamma / \rad^{\infty} = \ind \Gamma$$
\end{proof}

\begin{obs}
    One interpretation we can give of Theorem~\ref{th:gen st + mesh-comp} goes as follows: since the category $k(\Gamma)$ is in general more well-behaved than the category $\ind \Ga$, Lemma~\ref{lem:gen st + mesh-comp} suggests that mesh-comparability is associated to a certain "organization" of the morphisms between modules of $\Gamma$. So we can say that, for a component $\Gamma$ to be standard, it needs to have two characteristics: the first is that $\Gamma$ is mesh-comparable and so has this organization on the morphisms, the other being that the infinite radical $\rad^{\infty}$ vanishes in $\Gamma$, which means, because of Proposition~\ref{prop:parte inf e gen st}, that the morphisms do not have infinite parts. 
\end{obs}

\begin{obs}
    Theorem~\ref{th:gen st + mesh-comp} helps to answer the problem raised by Skowro{\'n}ski in \cite{Sko2}: to describe the generalized standard components which are not standard. Our result states that the generalized standard components which are not standard are exactly the ones which are not mesh-comparable. 
    It would still be desirable to have a structural description of the mesh-comparable components, similar to what is done to regular components by means of Happel-Preisel-Ringel and Zhang's theorem (\cite{HPR} and \cite{Zhang}) or to semiregular components as studied by S. Liu (\cite{Liu2}).
\end{obs}

\begin{obs}
    Because all generalized standard stable tubes are standard (see, e.g., \cite{SiSk}), they are also mesh-comparable. However, the existence of mesh-comparability on non-standard tubes is still an open problem.
\end{obs}

\begin{cor}
    The following are equivalent for a mesh-comparable component $\Gamma$: 
    \begin{enumerate}
        \item[(a)] $\Gamma$ is standard.
        \item[(b)] $\Gamma$ is generalized standard.
        \item[(c)]  For each morphism $f$ between modules from $\Gamma$ we have $\ff_{\infty}(f) = 0$.
    \end{enumerate}
\end{cor}

\begin{proof}
    It follows from Proposition~\ref{prop:parte inf e gen st} and Theorem~\ref{th:gen st + mesh-comp}.
\end{proof}

The next result also follows immediately from Theorem~\ref{th:gen st + mesh-comp}:

\begin{cor}
\label{cor:gen st n st is not mesh-c}
    If a component $\Gamma$ is generalized standard but not standard, then it is not mesh-comparable.
\end{cor}

Our next example builds on Corollary~\ref{cor:gen st n st is not mesh-c} to exhibit a component which is not mesh-comparable.

\begin{ex}[\cite{BG}, `Example 14 bis']  \label{ex:14bis}
    Let $A$ be the path $k$-algebra given by the quiver

    \begin{displaymath}
        \xymatrix{\bullet \ar@<0.5ex>[r]^{\delta} & \bullet \ar@<0.5ex>[l]^{\sigma} \ar@(ur,dr)[]^{\rho}}
    \end{displaymath}

bound by the relations  $\delta \sigma = \rho^2, \rho^4 = 0, \sigma \delta  = \sigma \rho \delta$. Then we know by the techniques from \cite{BG} that $A$ is representation-finite, so the single component $\Gamma$ of the Auslander-Reiten quiver of $A$ is generalized standard (by Auslander's theorem), and if the characteristic of the field $k$ is 2, that $\Gamma$ is not standard. Therefore, Corollary~\ref{cor:gen st n st is not mesh-c} implies that $\Gamma$ is not mesh-comparable.

Note that, by \cite{BGRS}, all finite Auslander-Reiten components are standard if $k$ has characteristic different from 2, so we can only expect to find examples of finite non-mesh-comparable components when the characteristic of the base field is 2.
\end{ex}

In the next example, we show that there are mesh-comparable components which are not generalized standard, and thus not standard. 

\begin{ex}
   Let $H$ be a wild hereditary algebra and let $\Gamma$ be a regular component of its Auslander-Reiten quiver. Then $\Gamma$ is not generalized standard, but it has the form $\mathbb{Z} \mathbb{A}_{\infty}$ (see, e.g., the work of Kerner \cite{Ker} for a reference on this). However, from Proposition~\ref{prop:easy exs mesh-c} above we know that $\Gamma$ is mesh-comparable.
\end{ex}

\section{Mesh-comparability and degrees}
\label{sec:mesh-c x degrees}

The notion of degree of an irreducible morphism was introduced by S. Liu in \cite{Liu1} as a tool to study compositions of irreducible morphisms and combinatorial aspects of the Auslander-Reiten quivers.

\begin{defi} We say that the {\bf left degree} of an irreducible morphism $f : X \rightarrow Y$
is $n$, and we denote $d_l(f) = n$, if $n$ is the smallest natural number for which there are an indecomposable module $Z$ and a
morphism $h : Z \rightarrow X$ such that $h \in \rad^n(Z,X) \setminus \rad^{n+1}(Z,X)$ and $fh \in \rad^{n+2}(Z, Y)$. In case this condition is not satisfied for any $n \geq 1$, then we say that the left degree of $f$ is infinite. The {\bf right degree} of $f$, denoted by $d_r(f)$, is defined dually.
\end{defi}

It was also shown in \cite{Liu1}, Corollary 1.7 that two irreducible morphisms $f,g: X \rightarrow Y$ representing the same arrow in the Auslander-Reiten quiver share the same left and right degrees. As a consequence, if we know the degrees of all the chosen irreducible morphisms over a mesh-comparable component, then we will immediately know the degrees of all irreducible morphisms, not only the chosen ones.

The degrees of chosen irreducible morphisms can be computed using our next theorem. Once we restrict ourselves to the context of mesh-comparable components, it is an extension of Theorem C from \cite{CMT1} and Theorem 3.3 from \cite{CSi1}.

\begin{teo}
    Let $\Gamma$ be an Auslander-Reiten component, which is mesh-comparable via $F$. Given a chosen irreducible morphism $f: X \rightarrow Y$, we have:

    \begin{enumerate}
        \item[(a)] if $d_l(f) = n < \infty$, then $\operatorname{Ker} f \text{ belongs to } \Gamma$ and there exists a path  $X_0 = \operatorname{Ker} f \xrightarrow{g_1} X_1 \xrightarrow{g_2} \cdots \xrightarrow{g_n} X_n = X$  of chosen irreducible morphisms such that $g_n\cdots g_1 \neq 0$ and $f g_n \cdots g_1 = 0$.

        \item[(b)]  if $d_r(f) = n < \infty$, then  $\operatorname{Coker} f \text{ belongs to }\Gamma$ and there exists a path  $X_0 = Y \xrightarrow{g_1} X_1 \xrightarrow{g_2} \cdots \xrightarrow{g_n} X_n = \operatorname{Coker} f$ of chosen irreducible morphisms such that $g_n\cdots g_1 \neq 0$ and $g_n \cdots g_1 f = 0$.
    \end{enumerate}
\end{teo}

\begin{proof}
   We shall prove only (a) since the proof of (b) is dual.  
   
   It follows from \cite{CMT1}, 3.4 that the canonical inclusion morphism  $\ker f: \Ker f \rightarrow X$ belongs to  $\rad^n(\Ker f, X)$ but not to $\rad^{n+1}(\Ker f, X)$, so $\Ker f \text{ belongs to } \Gamma$. 
   Since $f$ is a chosen morphism, we write $f = F(\overline{\gamma})$, where $\gamma: X \rightarrow Y$ is an arrow of $\Ga$. Moreover, since $\Dp(\ker f) = n$, we have  $\ff_n(\ker f) \neq 0$, and so, we can obtain a path $\Ker f \xrightarrow{\alpha_1} X_1 \xrightarrow{\alpha_2} \ldots \xrightarrow{\alpha_n} X$ in $\Ga$ such that  $F(\overline{\alpha_n}) \ldots F(\overline{\alpha_1}) \neq 0$. Then take $g_i = F(\overline{\alpha_i})$ for each $1 \leq i \leq n$. We shall prove that $fg_n \ldots g_1 = 0$.

   Observe that there exists a morphism $u: X \rightarrow X$ such that either $F(\overline{\alpha_n}) \ldots F(\overline{\alpha_1}) = u \circ (\ker f)$ or $\ker f = u F(\overline{\alpha_n}) \ldots F(\overline{\alpha_1})$ (see for instance \cite{AC}, Lemma II.2.8). Since $F(\overline{\alpha_n}) \ldots F(\overline{\alpha_1}) \neq 0$, it follows from Corollary~\ref{cor:composicao_bem_comportada_chosen} that  $F(\overline{\alpha_n}) \ldots F(\overline{\alpha_1}) \in \rad^n(\Ker f, X) \setminus \rad^{n+1}(\Ker f, X)$. We already saw that $\ker f \in \rad^n(\Ker f, X) \setminus \rad^{n+1}(\Ker f, X)$, so in both cases for  $u$, we can deduce that it is an automorphism. Changing $u$ by  $u^{-1}$ if necessary, we can assume that $\ker f = u F(\overline{\alpha_n}) \ldots F(\overline{\alpha_1})$. It then follows that 
   \begin{align*}
       0 &= f \circ (\ker f) = f u F(\overline{\alpha_n}) \ldots F(\overline{\alpha_1}) \\
       &= F(\overline{\gamma}) (\ff_P (u) + \ff_S(u)) F(\overline{\alpha_n}) \ldots F(\overline{\alpha_1})
   \end{align*}

  Hence  $F(\overline{\gamma}) \ff_P (u) F(\overline{\alpha_n}) \ldots F(\overline{\alpha_1}) = -F(\overline{\gamma}) \ff_S (u) F(\overline{\alpha_n}) \ldots F(\overline{\alpha_1})$, where what is to the left of the equality is either zero or belongs to to  $\rad^{n+1} \setminus \rad^{n+2}$,  by Proposition~\ref{prop:depth of F(phi)}, and what is to the right belongs to  $\rad^{n+2}$. It follows then that $F(\overline{\gamma}) \ff_P (u) F(\overline{\alpha_n}) \ldots F(\overline{\alpha_1}) = 0$. Since $u$ is an automorphism over $X$ indecomposable, there exists $\lambda \in k \setminus \{0\}$ such that $\ff_P(u) = \lambda.\Id_X$. Therefore,  $F(\overline{\gamma}) (\lambda.\Id_X) F(\overline{\alpha_n}) \ldots F(\overline{\alpha_1}) = \lambda.F(\overline{\gamma}) F(\overline{\alpha_n}) \ldots F(\overline{\alpha_1}) = 0$, that means, $F(\overline{\gamma}) F(\overline{\alpha_n}) \ldots F(\overline{\alpha_1}) = fg_n \ldots g_1 = 0$.
\end{proof}

\section*{Acknowledgements}
This work is part of the PhD thesis of first named author, under supervision by the second named author (\cite{ChustTese}). The authors gratefully acknowledge financial support by S\~ao Paulo Research Foundation - FAPESP (grants \#2020/13925-6 and \#2022/02403-4), and by CNPq (grant Pq 312590/2020-2).

\end{document}